\documentclass[12pt]{article}
\usepackage{amsthm}
\usepackage{amsmath}
\usepackage{amssymb}
\usepackage{graphicx}
\usepackage{float}
\usepackage{amsfonts}
\usepackage{latexsym}
\usepackage{marvosym}
\usepackage{wasysym}

\usepackage[utf8]{inputenc}
\usepackage[english]{babel}

\newtheorem{theorem}{Theorem}[section]
\newtheorem{corollary}[theorem]{Corollary}
\newtheorem{lemma}[theorem]{Lemma}

\newtheorem{proposition}[theorem]{Proposition}

\DeclareMathOperator{\boundary}{boundary}
\DeclareMathOperator{\interior}{interior}

\title{Intersection Graphs of Maximal Sub-polygons of $k$-Lizards}
\author{Caroline Daugherty \\ \small{Operations Research Center} \\ \small{Massachusetts Institute of Technology} \\ \small{Cambridge, MA} \and  Joshua D. Laison \\ \small{Mathematics Department} \\ \small{Willamette University}  \\ \small{Salem, OR } \and Rebecca Robinson \\ \small{Mathematical and Statistical Sciences Department} \\ \small{University of Colorado Denver} \\ \small{Denver, CO} \\ \and Kyle Salois \\ \small{Mathematics Department} \\ \small{Colorado State University}  \\ \small{Fort Collins, CO}}

\begin{document}

\maketitle
\begin{abstract}
We introduce $k$-maximal sub-polygon graphs ($k$-MSP graphs), the intersection graphs of maximal polygons contained in a polygon with sides parallel to a regular $2k$-gon.  We prove that all complete graphs are $k$-MSP graphs for all $k>1$; trees are $2$-MSP graphs; trees are $k$-MSP graphs for $k>2$ if and only if they're caterpillars; and $n$-cycles are not $k$-MSP graphs for $n>3$ and $k>1$.  We derive bounds for which $j$-cycles appear as induced subgraphs of $k$-MSP graphs.  As our main result, we construct examples of graphs which are $k$-MSP graphs and not $j$-MSP graphs for all $k>1$, $j>1$, $k \neq j$.
\end{abstract}



\section{Introduction}

Let $S$ be a set of geometric objects in the plane, and construct a graph $G$ with a vertex for every object in $S$, and an edge between two vertices if and only if their corresponding objects intersect.  We say that $S$ is an \textit{\textbf{intersection representation}} of $G$, and $G$ is an \textit{\textbf{intersection graph}}.    Intersection graphs of many different geometric objects have been studied \cite{bessy2020independent,hlinveny2001representing,mckee1999topics,pach2020almost}.

Recall that a \textit{\textbf{polyomino}} is the union of unit squares identified along their edges, as shown on the left in Figure~\ref{2-MSP-bull} \cite{golomb1996polyominoes}. In 1981 Berge, Chen, Chv\'{a}tal, and Seow first defined the intersection graph $G(P)$ of maximal axis-aligned rectangles in a polyomino $P$ \cite{berge1981combinatorial}.  They proved that if $P$ is semiconvex, then $G(P)$ is a comparability graph, and asked if $G(P)$ is always perfect.  Shearer answered this question in the affirmative in 1982 \cite{shearer}, and Maire provided an alternate proof in 1993 \cite{maire}.

In this paper we loosen the restriction of integer side lengths, and generalize from polygons with axis-aligned edges to polygons with edges chosen from $k$ distinct slopes, as follows.  We say the \textit{\textbf{direction}} of a line segment in the plane is the angle it makes with the $x$-axis.  For an integer $k>1$, an \textit{\textbf{allowed}} direction with respect to $k$ is a direction in the set $\{ i\pi/k \,|\, 0 \leq i < k \}.$

For an integer $k>1$, a \textit{\textbf{k-lizard}} is a polygon whose sides all have allowed directions with respect to $k$.  Equivalently, a $k$-lizard is a polygon whose sides are all parallel to the sides of a regular $2k$-gon.  This definition generalizes the \textit{\textbf{k-snakes}} defined in \cite{church2008snakes}, which also have integer side lengths. The $2$-snakes are polyominoes and the $3$-snakes are \textit{\textbf{polyiamonds}}, the union of unit equilateral triangles identified along their edges \cite{golomb1996polyominoes}.  In Section~\ref{open questions section} we briefly discuss our choice to investigate $k$-lizards instead of $k$-snakes.

The angle between two sides of a $k$-lizard is a member of the set $\{0, \pi/k, 2\pi/k, \ldots, (2k-1)\pi/k\}$.  We use the shorthand $\theta_i$ for the angle $i\pi/k$.  If the measure of angle $\alpha$ is $\theta_i$ and $i<k$, then $\alpha$ is convex, and if $i>k$ then $\alpha$ is reflex.  We don't consider angles of $\theta_k= \pi$ to be vertices of a polygon.

A convex $k$-lizard $S \subseteq L$ is a \textit{\textbf{maximal convex sub-polygon}}, or \textit{\textbf{scale}}, of $L$ if whenever $T$ is a convex $k$-lizard with $S \subseteq T \subseteq L$, $S=T$.  Figure~\ref{2-MSP-bull} shows a $2$-lizard with 5 scales, and a $3$-lizard with 6 scales.

A graph $G$ is a \textit{\textbf{$k$-maximal sub-polygon graph}} or \textit{\textbf{$k$-MSP graph}} if there exists a $k$-lizard $L$ and a bijection between the vertices of $G$ and the scales of $L$, such that two vertices are adjacent in $G$ if and only if the interiors of their corresponding scales intersect.  Figure~\ref{2-MSP-bull} shows examples of a 2-MSP graph and a 3-MSP graph.  Note that every $k$-lizard has exactly one corresponding $k$-MSP graph, but a $k$-MSP graph has many corresponding $k$-lizards.

\begin{figure}[ht]
\begin{center}
\includegraphics[width=.4\textwidth]{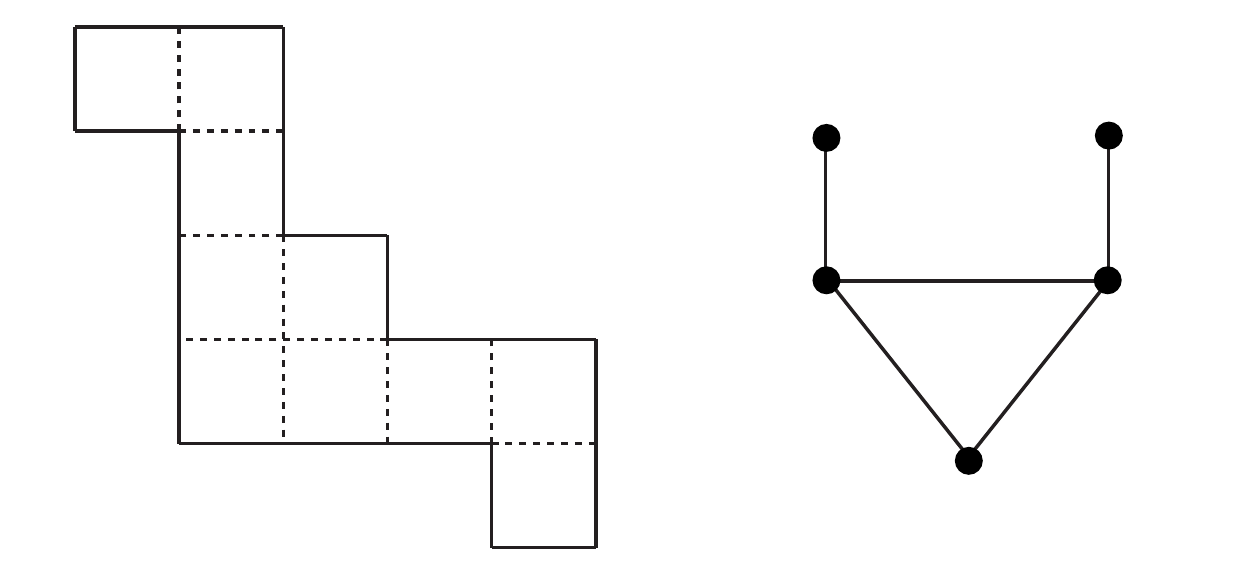} \includegraphics[width=.55\textwidth]{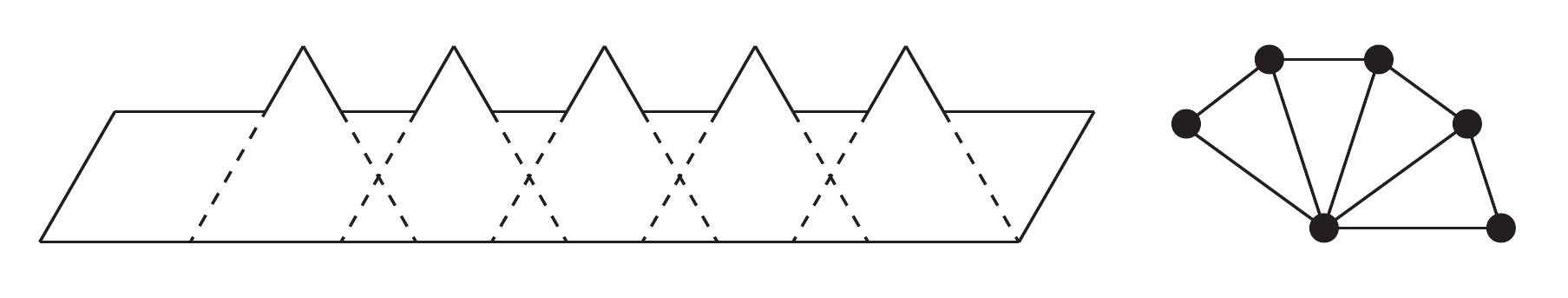}
\end{center}
\caption{A $2$-lizard and its $2$-MSP graph, and a $3$-lizard and its $3$-MSP graph.}
\label{2-MSP-bull}
\end{figure}

\section{Geometric Properties of $k$-Lizards}

In this section we derive geometric tools for analyzing $k$-lizards which we use to prove results about $k$-MSP graphs in Section~\ref{families section} and Section~\ref{separating examples}.

Suppose $L$ is a $k$-lizard.  We define a \textbf{\textit{proto-scale}} of $L$ to be a line segment in an allowed direction, contained in $L$ and containing a reflex angle of $L$ in its interior.

\begin{lemma}\label{Convex Subset Lemma}
Let $L$ be a $k$-lizard. Suppose $R$ is  a line segment contained in $L$ in an allowed direction whose intersection with $\boundary(L)$ is either empty, a vertex of $L$, or contained in an edge of $L$. Then $R$ is contained in at least one scale of $L$.  In particular, every proto-scale of $L$ is contained in a scale of $L$.
\end{lemma}

\begin{proof}
Suppose $R$ is a line segment in an allowed direction whose intersection with $\boundary(L)$ is a vertex of $L$ or contained in an edge of $L$.  We can thicken $R$ by a small amount in a different allowed direction, to create a parallelogram $P$ containing $R$ and contained in $L$. Since $P$ is a convex $k$-lizard, by definition $P$ is contained in a scale of $L$, so $R$ is also contained in that scale.
\end{proof}

\begin{lemma} \label{reflex lemma}
Let $L$ be a $k$-lizard.  If two scales $A$ and $B$ of $L$ intersect in their interiors, then the polygon $A\cup B$ has at least one reflex angle.
\end{lemma}

\begin{proof}
If the polygon $A\cup B$ has no reflex angles, then it's convex, and so by definition $A\cup B$ is contained in some scale $C$. This contradicts the fact that $A$ and $B$ are maximal convex $k$-lizards in $L$.
\end{proof}

\begin{figure}[ht]
\begin{center}
\includegraphics[width=.35\textwidth]{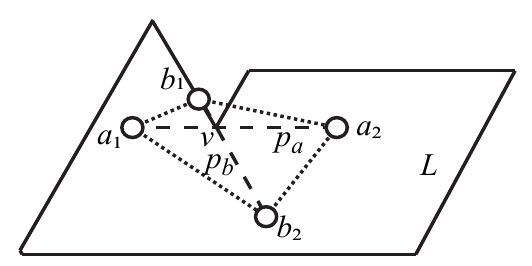}
\end{center}
\caption{Distinct proto-scales in the proof of Lemma~\ref{proto-scales-lemma}.}
\label{extending-scale-boundary} \label{distinct-scales-fig}
\end{figure}

\begin{lemma}
For a proto-scale $x$ and a scale $S$, if $x \subset S$ then $x \subset \boundary{(S)}$.
\end{lemma}

\begin{proof}
Suppose $x$ contains the reflex angle $q$ of $L$ in its interior.  Since $q \in x \subset S \subset L$ and $q \in \boundary(L)$, $q \in \boundary(S)$.  Suppose by way of contradiction that $x \not\subset \boundary(S)$.  Then $x$ contains a point $r$ in the interior of $S$.  Since $q$ is in the interior of $x$ and $S$ is convex, points in $x$ on the other side of $r$ from $q$ are not in $S$, contradicting our hypothesis.
\end{proof}

\begin{lemma}%[Proto-scales Lemma]
\label{proto-scales-lemma}
Suppose $v$ is a vertex of a $k$-lizard $L$ with interior angle measure $\theta_{k+j}$, for an integer $j \geq 1$. Then there is a point $x$ in $L$ near $v$ contained in the interiors of at least $j+1$ scales.
\end{lemma}

\begin{proof}
Let $p_a$ and $p_b$ be proto-scales containing $v$ with distinct directions, with endpoints $a_1$, $a_2$, $b_1$, and $b_2$, respectively, as shown on the right in Figure~\ref{distinct-scales-fig}.  Construct the quadrilateral $Q=a_1b_1a_2b_2$.  Since $p_a$ and $p_b$ are diagonals of $Q$ and $v$ is the intersection of $p_a$ and $p_b$, $v$ is in the interior of $Q$.  Since $v$ is on the boundary of $L$, $Q$ contains points outside of $L$.  Thus scales containing $p_a$ and $p_b$ are distinct.

There are $j+1$ proto-scales at $v$, and each one is contained in a scale of $L$ by Lemma~\ref{Convex Subset Lemma}.  So $v$ is in the intersection of at least $j+1$ scales $S_1$, $\ldots$, $S_{j+1}$.  By construction, each scale $S_i$ contains $v$ in the interior of one of its edges.  Consider the half-planes defined by these edges and containing the scales $S_1$, $\ldots$, $S_{j+1}$.  Since all of these half-planes contain $v$ in their boundary, they have a common intersection, which also contains $v$.  Points in this intersection near $v$ are contained in the interiors of all the scales $S_1$, $\ldots$, $S_{j+1}$.
\end{proof}

\begin{lemma} \label{two-proto-scales-lemma}
Suppose two sides $s_1$ and $s_2$ of a $k$-lizard $L$ lie on the same line $\ell$, and the interior of the line segment $s_3$ between $s_1$ and $s_2$ on $\ell$ is in $\interior(L)$.  Say the endpoint of $s_1$ on $s_3$ is $v_1$, the endpoint of $s_2$ on $s_3$ is $v_2$, and the interior angle measures of $v_1$ and $v_2$ are $\theta_{k+j}$ and $\theta_{k+m}$, for integers $j \geq 1$ and $m \geq 1$. Then there is a point $x$ in $L$ near $v_1$ contained in the interiors of at least $jm+1$ scales.
\end{lemma}

\begin{figure}[ht]
\begin{center}
\includegraphics[width=.3\textwidth]{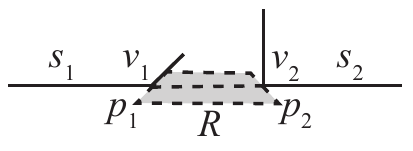}
\end{center}
\caption{Constructing the parallelogram $R$ in the proof of Lemma~\ref{two-proto-scales-lemma}.}
\label{two-proto-scales-figure}
\end{figure}

\begin{proof}
Suppose the hypotheses of the lemma, as shown in Figure~\ref{two-proto-scales-figure}.  By the proof of Lemma~\ref{proto-scales-lemma}, there are proto-scales at $v_1$ in $j+1$ different directions and proto-scales at $v_2$ in $m+1$ different directions.  One of the directions at $v_1$ is along $\ell$ and one of the directions at $v_2$ is along $\ell$.  Let $p_1$ be a proto-scale at $v_1$ not along $\ell$ and let $p_2$ be a proto-scale at $v_2$ not along $\ell$.  Since $v_1v_2$ is in an allowed direction, we may choose $p_1$ and $p_2$ small enough so that a parallelogram $R$ with two sides $p_1$ and $p_2$ and two sides parallel to $v_1v_2$ is contained in $L$.  By Lemma~\ref{Convex Subset Lemma}, $R$ is contained in a scale of $L$, and by the proof of Lemma~\ref{proto-scales-lemma}, this scale is distinct from scales formed from different proto-scales at $v_1$ and $v_2$.  Thus there are at least $jm$ of these scales, and at least one additional scale with side along $\ell$ containing $v_1$.  We can find a point near $v_1$ contained in all of them.
\end{proof}

\begin{lemma}%[Ray Lemma]
\label{Ray Lemma}
Given three vertices $a$, $b$, and $c$ in a $k$-MSP graph $G$ with corresponding scales $A$, $B$, and $C$ such that $abc$ forms an induced path in $G$, if from every point on $\boundary(B-A)$ there exists a ray in an allowed direction not along $\boundary(B-A)$ that intersects $A$, then there is a scale $D$ intersecting the interiors of $A$, $B$, and $C$.
\end{lemma}

\begin{proof}
Let scales $A$, $B$, and $C$ be in a $k$-lizard $L$ as described above. Since $\interior(B)$ and $\interior(C)$ intersect, let $p$ be a point on the boundary of $B$ in the interior of $C$, as shown in Figure~\ref{extending-scale}. Since $A$ and $C$ are not adjacent in $G$, $p \not\in A$.  Hence $p \in \boundary(B-A)$, and so by hypothesis there is a ray $r$ from $p$ in an allowed direction, not along $\boundary(B-A)$, that intersects the interior of $A$.  Choose a point $q \in r \cap \interior(A)$ such that the line segment $pq$ is contained in the interior of $L$.  Since $r$ is not along $\boundary(B-A)$, $pq$ also contains points in $\interior(B)$. Since $pq$ is in an allowed direction and contained in the interior of $L$, by Lemma~\ref{Convex Subset Lemma} it must be contained in some scale $D$ of $L$, and since $pq$ intersects the interiors of $A$, $B$, and $C$, $D$ intersects the interiors of $A$, $B$, and $C$.
\end{proof}

\begin{figure}[ht]
\begin{center}
\includegraphics[width=.43\textwidth]{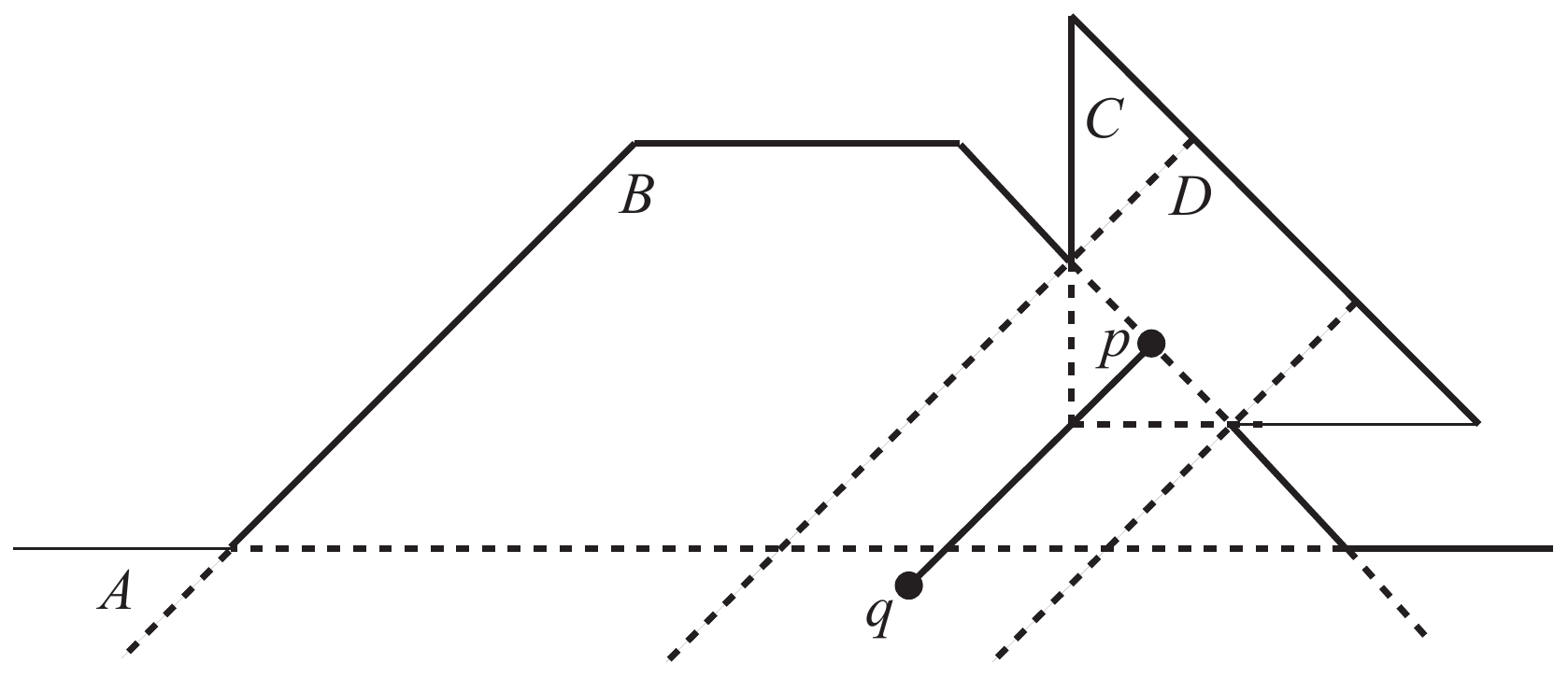} \hspace*{1cm} \includegraphics[width=.47\textwidth]{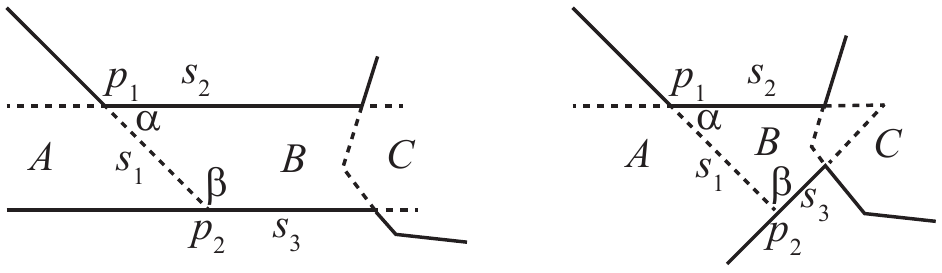}
\end{center}
\caption{On the left, constructing the scale $D$ in the proof of Lemma~\ref{Ray Lemma}.  In the center and right, extending the line segment $s_1$ to form the angle $\beta$ in the proof of Lemma \ref{parallel sides lemma}.}
\label{extending-scale}
\end{figure}

Suppose $A$ and $B$ are two scales of a $k$-lizard $L$ that intersect in their interiors.  Since $A$ and $B$ are convex, $\boundary(A)$ and $\boundary(B)$ intersect in exactly two points $p_1$ and $p_2$.  We say the sides of $B$ \textit{\textbf{incident}} with $A$ are the sides containing $p_1$ and $p_2$, and not contained in $A$.

\begin{lemma}%[Parallel Sides Lemma]
\label{parallel sides lemma}
Let $L$ be a $k$-lizard, and $A$, $B$, and $C$ be scales in $L$ with corresponding vertices $a$, $b$, and $c$ in the $k$-MSP graph $G$ of $L$. Suppose $abc$ is a path in $G$, and no other vertex in $G$ is adjacent to both $a$ and $b$.  Then the sides of $B$ incident with $A$ are parallel.
\end{lemma}

\begin{proof}
Assume the hypotheses of the lemma.

Let $p_1$ and $p_2$ be the intersection points of $\boundary(A)$ and $\boundary(B)$.  The vertices of $A\cup B$ that are not also vertices of $A$ or vertices of $B$ must be $p_1$ or $p_2$. Therefore by Lemma~\ref{reflex lemma}, one of these vertices, say $p_1$, is reflex in $A\cup B$.

Since $p_1$ is reflex in $A \cup B$, it has an internal angle of at least $\theta_{k+1}$.  If it has an angle of greater than $\theta_{k+1}$, by Lemma~\ref{proto-scales-lemma} there are at least three scales containing $p_1$, contradicting our hypotheses.  So the angle of $A\cup B$ at $p_1$ is exactly $\theta_{k+1}$.

Again since $p_1$ is a reflex angle of $A \cup B$, one of the sides $s_1$ with endpoint $p_1$ in $A\cup B$ must be a side of $A$, and the other must be a side of $B$.  Extend $s_1$ past $p_1$ until it intersects $\boundary(B)$ again, as shown in Figure~\ref{extending-scale}.  If this line segment goes outside of $A$, it's contained in a third scale intersecting $A$ and $B$, contradicting our hypotheses.  So $s_1$ intersects $\boundary(B)$ a second time at $p_2$.

Call the sides of $B$ incident with $A$ $s_2$ and $s_3$.  The small angle $\alpha$ between $s_1$ and $s_2$ is $\theta_1$ since $L$ has an angle of $\theta_{k+1}$ at $p_1$.  To show that $s_2$ and $s_3$ are parallel, we show that the angle $\beta$ between $s_1$ and $s_3$ is $\theta_{k-1}$.  Since $B$ is convex and $s_1$ is contained in $B$, $\beta$ is convex, and $\beta \leq \theta_{k-1}$.

Suppose by way of contradiction that $\beta \leq \theta_{k-2}$.  Then the triangle formed by extending $s_2$ and $s_3$ contains $B-A$, as shown on the right in Figure~\ref{extending-scale}.  Then the hypotheses of Lemma~\ref{Ray Lemma} are satisfied, since we can choose a ray from each point in this triangle parallel to either $s_2$ or $s_3$ that avoids $\boundary(B-A)$ and intersects $A$.  Therefore by Lemma~\ref{Ray Lemma} there's a scale intersecting the interiors of $A$, $B$, and $C$, contradicting our hypotheses.
\end{proof}

\begin{lemma}
\label{parallel sides corollary}
Let $L$ be a $k$-lizard, and $A$, $B$, and $C$ be scales in $L$ with corresponding vertices $a$, $b$, and $c$ in the $k$-MSP graph $G$ of $L$. Suppose $abc$ is an induced path in $G$,  no other vertex in $G$ is adjacent to both $a$ and $b$, and no other vertex in $G$ is adjacent to both $b$ and $c$.  Then the parallel sides of $B$ incident with $A$ and the parallel sides of $B$ incident with $C$ are the same.
\end{lemma}

\begin{proof}
Assume the hypotheses of the lemma, $s_1$ and $s_2$ are the parallel sides of $B$ incident with $A$, $s_3$ and $s_4$ are the parallel sides of $B$ incident with $C$, and suppose by way of contradiction that $s_1$, $s_2$, $s_3$, and $s_4$ are distinct.  Since $B$ is convex, these four sides are without loss of generality in the order $s_1$, $s_3$, $s_2$, $s_4$ counterclockwise around $B$.  So again without loss of generality, $s_3 \subseteq A$, and likewise $s_1 \subseteq C$, as shown in Figure~\ref{parallel corollary figure}.

Say the endpoint of $s_3$ on $C$ is $p_1$. Since $s_3 \subseteq A$, $p_1 \in A \cap C$.  Since $\interior(A)$ and $\interior(C)$ don't intersect, $p_1$ is on the boundary of both $A$ and $C$.  If $p_1$ is in the interior of $L$ or is a convex angle of $L$, we can construct another scale intersecting both $A$ and $C$, contradicting the hypotheses of the lemma.  So $p_1$ is a reflex angle of $L$.  By Lemma~\ref{proto-scales-lemma}, the internal angle of $L$ at $p_1$ is $\theta_{k+1}$ since no other scales intersect all of $A$, $B$, and $C$.

Call the other endpoint of $s_1$ $p_2$ and the other endpoint of $s_3$ $p_3$.  By similar reasoning, $p_2$ and $p_3$ must be on the boundary of $L$.  If $L$ had a reflex angle at $p_2$, then there would be more than two scales formed from the proto-scales at $p_1$ and $p_2$. So $L$ has an angle of at most $\theta_k$ at $p_2$, and similarly at $p_3$.  Then $A$ and $C$ extend past $s_3$ and $s_1$ and intersect each other, contradicting our hypotheses.
\end{proof}

In a $k$-lizard $L$, for a scale $A$ of $L$, we say that an \textit{\textbf{end $2$-region}} of $A$ is the intersection of $A$ with another scale $B$, such that $A-B$ is connected, and $\interior(A\cap B)$ is disjoint from other scales of $L$.  In Figure~\ref{end regions figure} on the right, $A\cap B$ is an end $2$-region of $A$, but $A \cap D$ is not since $A-D$ is disconnected.

\begin{figure}[ht]
\begin{center}
\includegraphics[width=.3\textwidth]{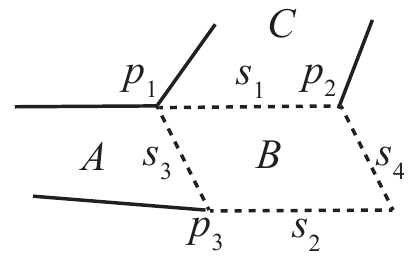} \hspace*{1cm}
\includegraphics[width=.55\textwidth]{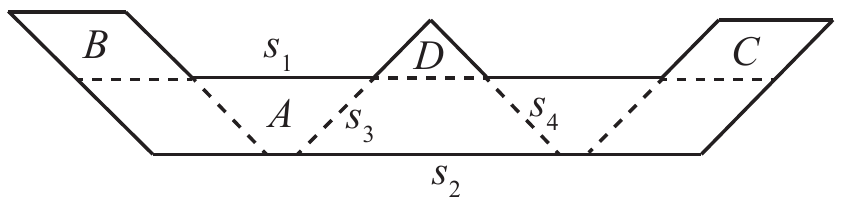}
\caption{On the left, the sides $s_1$ and $s_3$ in the proof of Lemma~\ref{parallel sides corollary}.  On the right, the intersection of $A$ and $D$ in the proof of Lemma~\ref{end regions lemma}.} \label{parallel corollary figure} \label{end regions figure}
\end{center}
\end{figure}

\begin{lemma}%[End 2-Regions]
\label{end regions lemma}
Any scale has at most two end 2-regions.
\end{lemma}

\begin{proof}
Let $A$ be a scale in a $k-$lizard $L$, and assume by way of contradiction that there are three scales $B$, $C$, and $D$ of $L$ intersecting $A$ in three end $2$-regions.

We claim that $\interior(B)$ and $\interior(C)$ don't intersect.  Suppose by way of contradiction that they do.  By hypothesis, the region $B\cap C$ is disjoint from the regions $A\cap B$ and $A\cap C$.  Let $x\in \interior(A \cap B)$, $y \in \interior(A \cap C)$, and $z \in \interior(B \cap C)$. The triangle $xyz$ is contained in $L$.  A scale containing this triangle is distinct from $A$ and $B$ and intersects $\interior(A \cap B)$, contradicting the assumption that $\interior(A \cap B)$ is a 2-region.

Then by Lemma~\ref{parallel sides corollary}, the region $A-(B\cup C)$ has parallel sides $s_1$ and $s_2$, which are the parts of the boundary of $A$ not contained in the interior of $B$ or $C$.

Since $D \cap A$ is disjoint from $B$ and $C$, $D$ intersects $\boundary(A)$ on $s_1$ or $s_2$.  Say $D$ intersects $A$ on $s_1$, and call these sides of $D$ $s_3$ and $s_4$, as shown on the right in Figure~\ref{end regions figure}.  Furthermore, the internal angles at these points in $D \cup A$
%formed by the sides of $D$ intersecting $s_1$, together with $s_1$,
must have measure $\theta_{k+1}$, otherwise by Lemma~\ref{proto-scales-lemma}, there would be another scale intersecting $A$ and $D$. Since $D$ is maximal in $L$ and the sides $s_3$ and $s_4$ are angled away from each other, they both intersect $s_2$.  Hence $A-D$ is disconnected, contradicting our assumptions.
\end{proof}

\section{Families of $k$-MSP Graphs} \label{families section}

In this section we prove that all complete graphs are $k$-MSP graphs for all $k$, trees are $k$-MSP graphs for all $k$ if and only if they're caterpillars, and no cycle is a $k$-MSP graph for any $k>3$.  We also investigate induced cycles in $k$-MSP graphs.

\subsection{Complete graphs}

\begin{figure}[H]
\begin{center}
\includegraphics[width=12cm]{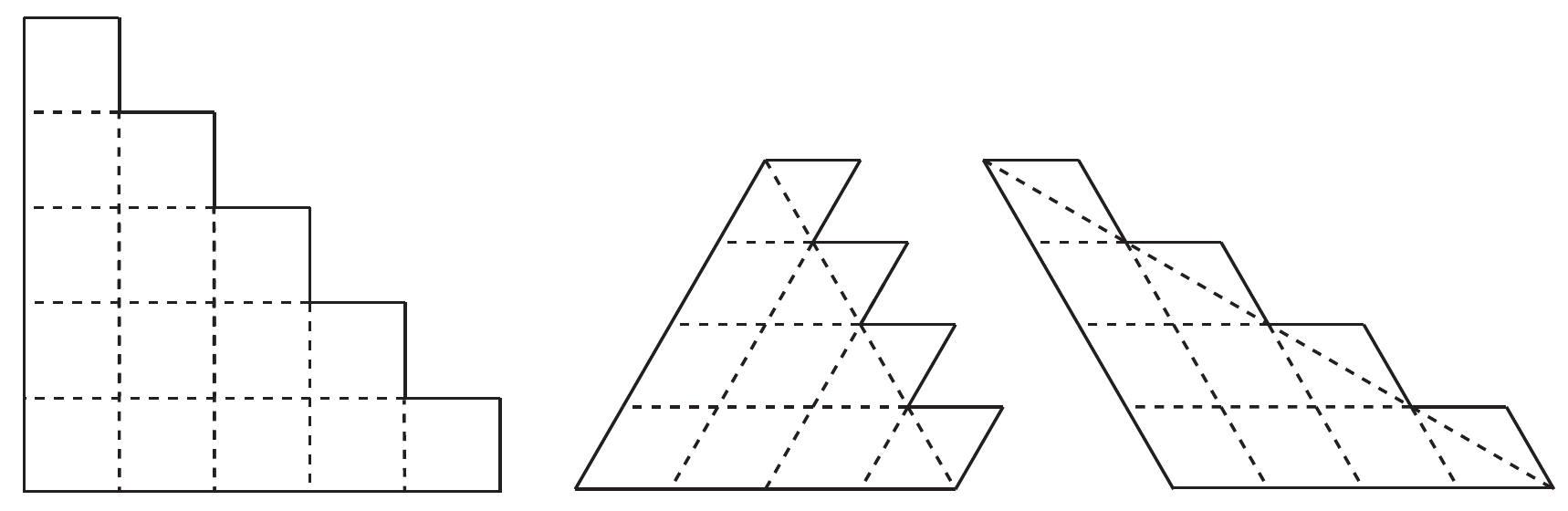}
\end{center}
\caption{Representations of $K_5$ as a $2$-MSP graph (left), a $3$-MSP graph (center), and a $6$-MSP graph (right).}
\label{complete}
\label{CompleteGraphs}
\end{figure}

\begin{proposition}
The complete graph $K_n$ is a $k$-MSP graph for all $k>1$ and $n \geq 1$.
\end{proposition}
\begin{proof}
For $k > 2$, we construct a $k$-lizard $L$ with $2n$ sides: two incident sides $s_1$ and $s_2$ with directions $\theta_1$ and $\theta_{k-1}$ respectively and both with length $n-1$, and $2n-2$ unit sides alternating between directions $\theta_{k-1}$ and $\theta_1$.  Examples of this construction with $n=5$ and $k=3$ and $k=6$ are shown in Figure~\ref{complete}.

Each of the steps create individual maximal sub-polygons that all contain the region in $L$ near the corner between $s_1$ and $s_2$. The triangular region bounded on two sides by $s_1$ and $s_2$ forms an additional maximal sub-polygon, also containing this region. Thus all scales of $L$ intersect in their interiors, and $L$ has $K_n$ as its $k$-MSP graph.

For $k=2$, we modify the construction so that $s_2$ has direction $\theta_2$ and $L$ has $2n+2$ total sides instead of $2n$. The first polygon in Figure~\ref{complete} is an example of this construction with $n=5$.  This has the same maximal sub-polygons as the construction above, without the corner triangle, and again they all overlap.

Therefore $K_n$ is a $k$-MSP graph for all $k$.
\end{proof}

\subsection{Trees}

Theorems~\ref{2 tree theorem} and~\ref{caterpillar theorem} below completely characterize which trees are $k$-MSP graphs for any $k \geq 2$.

\begin{theorem}
Every tree is a $2$-MSP graph. \label{2 tree theorem}
\end{theorem}

\begin{proof}
We prove a stronger statement by induction.  We prove that every tree has a $2$-MSP representation $L$ in which each scale contains a rectangle that contains no other scale. For the base case, a rectangle has only one scale, and has $2$-MSP graph $K_1$.
For the induction step, suppose $T$ is a tree with $n$ vertices, let $x$ be a leaf of $T$, and let $z$ be the neighbor of $x$ in $T$.  By induction, $T-x$ has a $2$-MSP representation $L$ in which the scale $Z$ corresponding to the vertex $z$ contains a rectangle $R$ that contains no other scale.  We replace $R$ with the Y-shaped construction $Y$ shown in Figure~\ref{tree figure}.  This is always possible since there's always a region next to $R$ of some small width outside of $L$.  This new 2-lizard has one new scale for $x$ that intersects $Z$ and no other scale, so it's a 2-MSP representation of $T$.
\end{proof}

\begin{figure}[ht]
\begin{center}
\includegraphics[width=10cm]{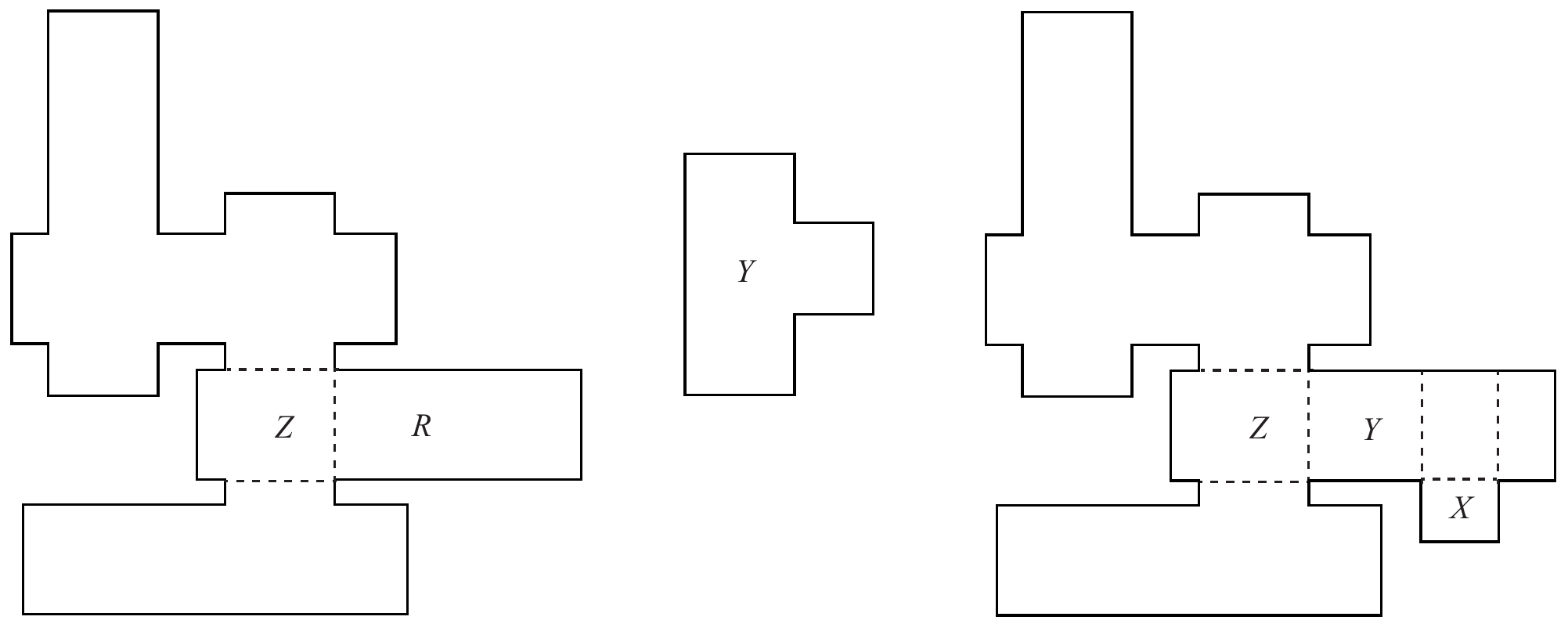}
\end{center}
\caption{The induction step in the proof of Theorem~\ref{2 tree theorem}.}
\label{tree figure}
\end{figure}

\begin{figure}[ht]
\begin{center}
\includegraphics[width=0.5\textwidth]{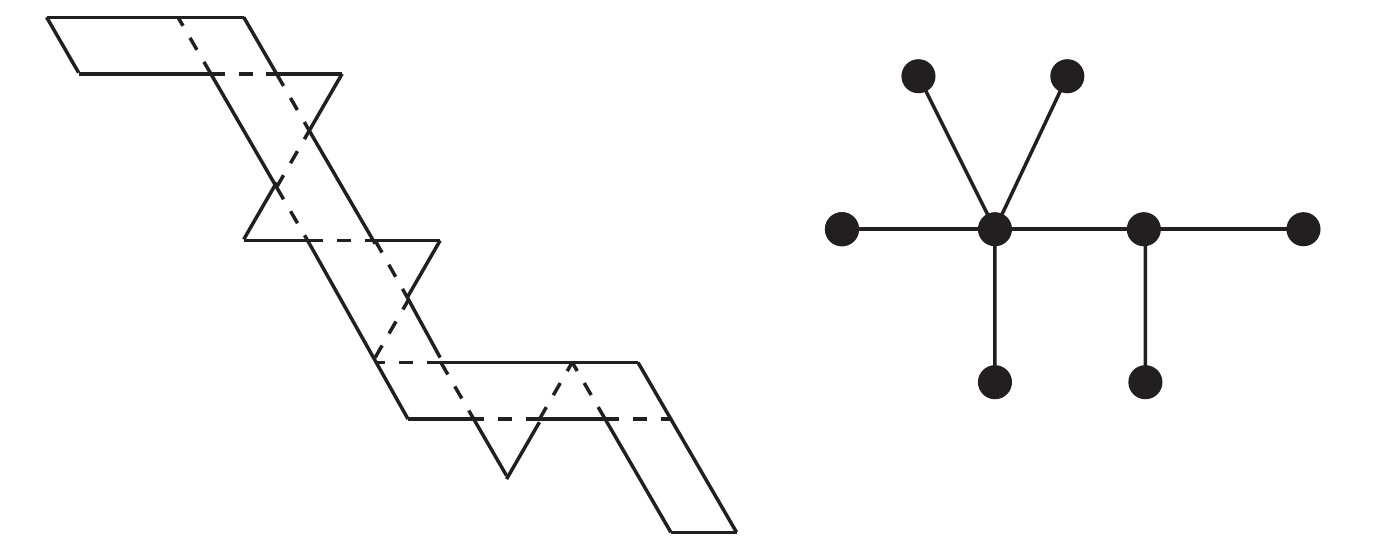} \hspace*{1cm}
\includegraphics[width=0.35\textwidth]{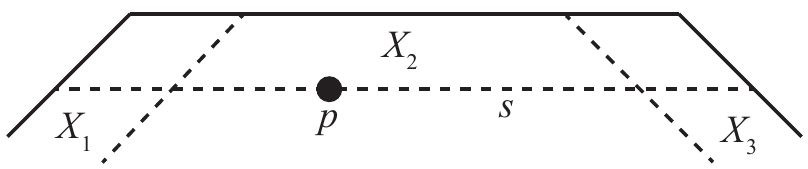}
\end{center}
\caption{On the left, a $3$-MSP representation of a caterpillar.  On the right, the point $p$ in the proof of Theorem~\ref{cycle-theorem}.}
\label{caterpillar} \label{cycle-proof-figure}
\end{figure}

Recall that a \textit{\textbf{caterpillar}} graph is a tree in which a single path (the \textit{\textbf{spine}}) is incident to or contains every edge of the tree \cite{west2001introduction}.

\begin{theorem}
All caterpillars are $k$-MSP graphs for all $k \geq 3$.  No non-caterpillar trees are $k$-MSP graphs for any $k \geq 3$. \label{caterpillar theorem}
\end{theorem}

\begin{proof}
To construct a caterpillar $G$ as a $k$-MSP graph, we first construct the spine using skinny parallelograms meeting at alternating $\theta_{k+1}$ and $\theta_{k-1}$ angles, as shown in Figure~\ref{caterpillar}.  Each parallelogram $A$ is a scale, representing a vertex $a$ on the spine of $G$.  We add triangles on each parallelogram $A$, one for each leaf of $G$ not on the spine incident to $a$, meeting the sides of $A$ at $\theta_{k+1}$ angles to ensure no additional scales are formed.  We choose $A$ long enough to fit these triangles so they don't overlap.
%Note that triangles require $k \geq 3$.

Conversely, we prove in Theorem~\ref{seagull-theorem} below that the graph $S_2$ shown in Figure~\ref{seagulls} is not a $k$-MSP graph for $k>2$.  Since every non-caterpillar tree contains $S_2$ \cite{west2001introduction}, no non-caterpillar trees are $k$-MSP graphs for any $k \geq 3$.
\end{proof}

\subsection{Cycles}

\begin{theorem}
No cycle $C_j$ is a $k$-MSP graph, for any $j \geq 4$ and for any $k \geq 2$. \label{cycle-theorem}
\end{theorem}

\begin{proof}
Suppose by way of contradiction that $L$ is a $k$-lizard with $k$-MSP graph $C_j$.  Label the vertices of $C_j$ by $x_1$, $\ldots$, $x_j$, and the corresponding scales of $L$ by $X_1$, $\ldots$, $X_j$.
%In what follows let $x_{j+1}=x_1$ and $X_{j+1}=X_1$.

Consider the sides of $X_2$ incident to $X_1$.  Since $X_2$ is not a cut vertex of $C_j$, one of these sides, say $s$, must be in the interior of $L$, as shown on the right in Figure~\ref{cycle-proof-figure}.  Since $\interior(X_1)$ and $\interior(X_3)$ are disjoint, there is a point $p$ on $s$ not in the interior of either $X_1$ or $X_3$.  Since $p \in \boundary(X_2)$ and $X_2$ intersects no scale other than $X_1$ and $X_3$, $p$ is not contained in the interior of any other scale of $L$.  But $p \in \interior(L)$, which is a contradiction.
\end{proof}

\subsection{Chordless cycles}

Recall that a \textit{\textbf{chordless $j$-cycle}} is an induced subgraph isomorphic to $C_j$ \cite{west2001introduction}. In this subsection we consider the question, for which positive integers $k$ and $j$ does there exist a $k$-MSP graph with a chordless $j$-cycle?  Shearer proved that no intersection graph of the maximal rectangles of a polyomino has a chordless $j$-cycle for $j>4$, and this proof is easily adapted to $2$-MSP graphs \cite{maire}.

\begin{theorem}[Shearer]
No $2$-MSP graph contains a chordless $j$-cycle for $j>4$.
\end{theorem}

Conversely, Figure~\ref{3MSP induced cycles figure} shows a 3-lizard whose 3-MSP graph has a chordless 5-cycle.  In fact, Figure~\ref{3MSP induced cycles figure} also shows a 3-lizard whose 3-MSP graph has a chordless 12-cycle.  By removing notches, we can construct a 3-MSP graph with any induced cycle between 4 and 12.  We don't know if a 12-cycle is the largest induced cycle in any 3-MSP graph.  We pose this as an open question at the end of this paper.

Figure~\ref{large induced cycles figure}  shows a 4-lizard whose 4-MSP graph has a chordless 16-cycle.  Again, we don't know whether this is the largest induced cycle in any 4-MSP graph. Figure~\ref{large induced cycles figure} also shows a 5-lizard whose 5-MSP graph has a chordless 10-cycle, and this construction can be extended to obtain 5-MSP graphs with arbitrarily large cycles.  This construction doesn't work for 3-MSP or 4-MSP graphs, but does for $k$-MSP graphs with $k>5$.  We state this formally in the following proposition, and summarize our findings on chordless cycles in $k$-MSP graphs in Table~\ref{chordless cycle table}.

\begin{proposition}
For any integers $j>3$ and $k>5$, there is a $k$-MSP graph with an induced $j$-cycle.
\end{proposition}

\begin{table} \begin{center}
\begin{tabular}{ c | c }
Graph family & Largest induced cycle \\
\hline \hline
2-MSP & 4 \cite{maire}\\
3-MSP & $\geq 12$ \\
4-MSP & $\geq 16$ \\
5-MSP & unbounded
\end{tabular}
\end{center}
\caption{The largest chordless cycle possible in families of $k$-MSP graphs.} \label{chordless cycle table}
\end{table}

\begin{figure}
\begin{center}
\includegraphics[width=0.7\textwidth]{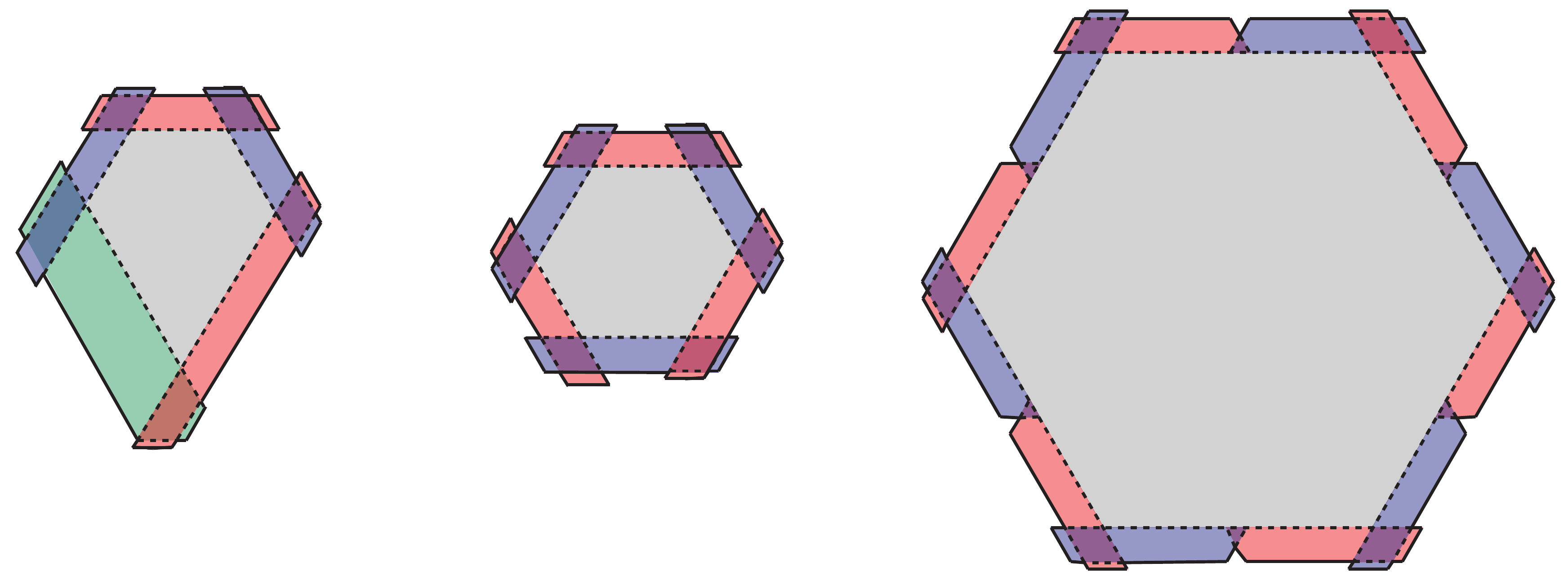}
\end{center}
\caption{Scales in 3-lizards that represent a 5-cycle, 6-cycle, and 12-cycle.}
\label{3MSP induced cycles figure}
\end{figure}

\begin{figure}
\begin{center}
\includegraphics[width=0.4\textwidth]{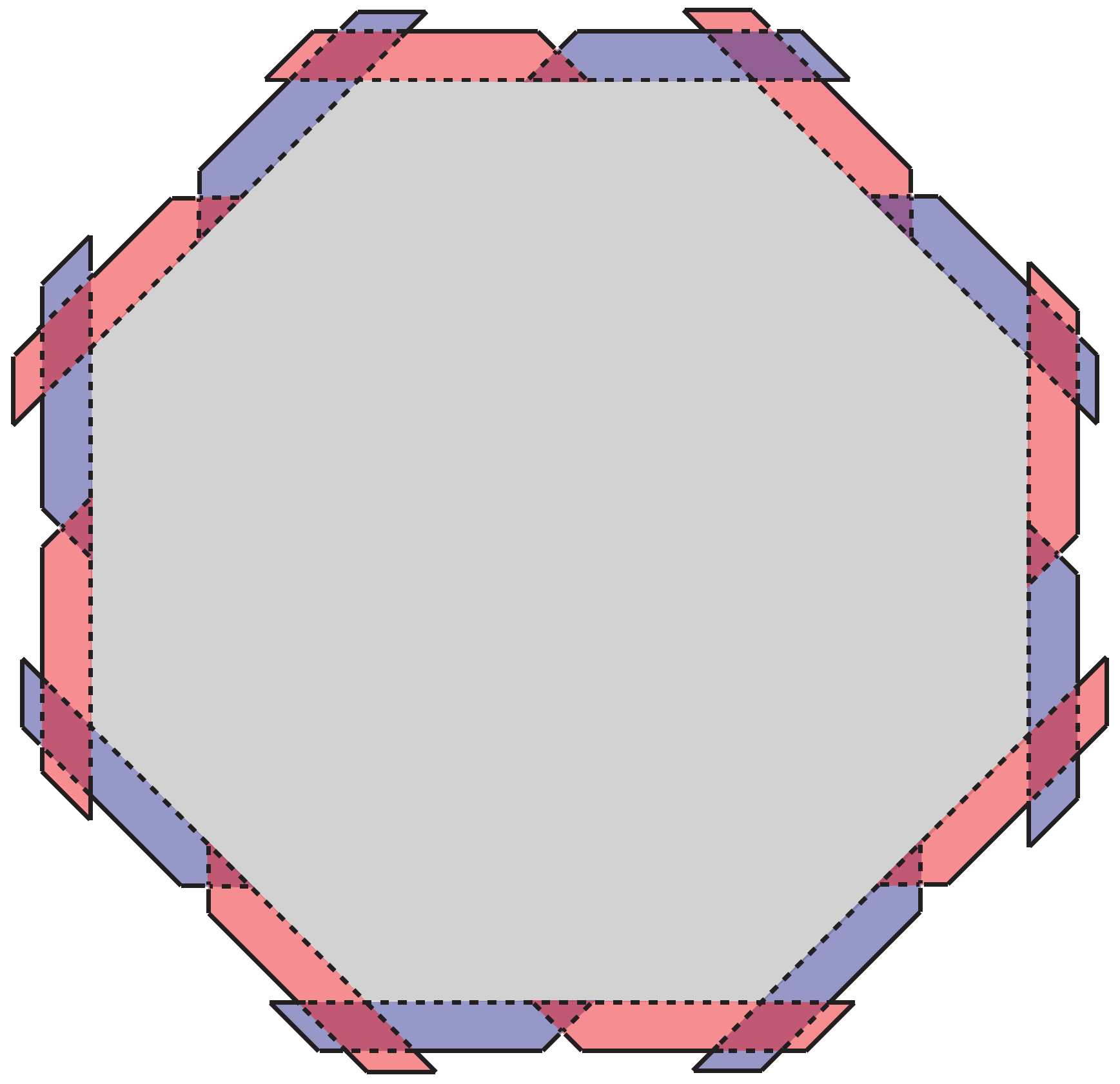} \hspace*{1cm} \includegraphics[width=0.4\textwidth]{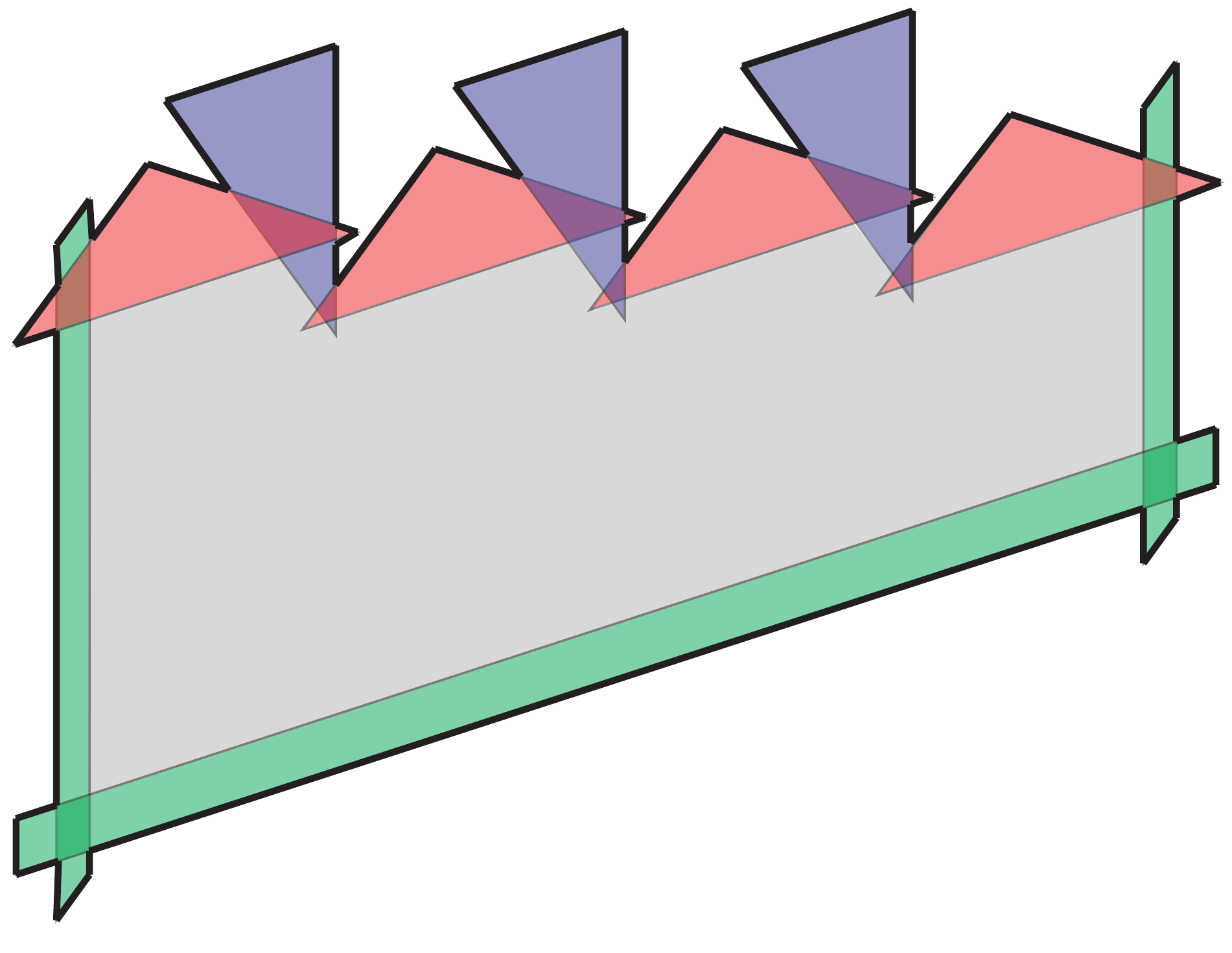}
\end{center}
\caption{On the left, scales in a 4-lizard that represent a 16-cycle.  On the right, scales in a 5-lizard that represent an arbitrarily large induced cycle.}
\label{large induced cycles figure}
\end{figure}
%%%%%

\section{Separating Examples}
\label{separating examples}

In this section we construct examples of graphs which are $k$-MSP graphs and not $j$-MSP graphs for all $k>1$, $j>1$, $k \neq j$.  We construct three families of graphs, the seagull graphs $S_k$, the turtle graphs $T_{n,k+1}$, and the seagull-turtle graphs $ST_{n,k}$.

We define the \textbf{\textit{seagull graph}} $S_n$ to be the complete graph $K_n$ with five additional vertices: two paths with two edges attached to one of the vertices of $K_n$, and one path with one edge attached to a different vertex of $K_n$.  We label the distinguished vertices of $S_n$ as $a$, $b$, $c$, $d$, $e$, $f$, and $g$, as shown in Figure~\ref{seagulls}.

\begin{figure}[ht]
\begin{center}
\includegraphics[width=0.6\textwidth]{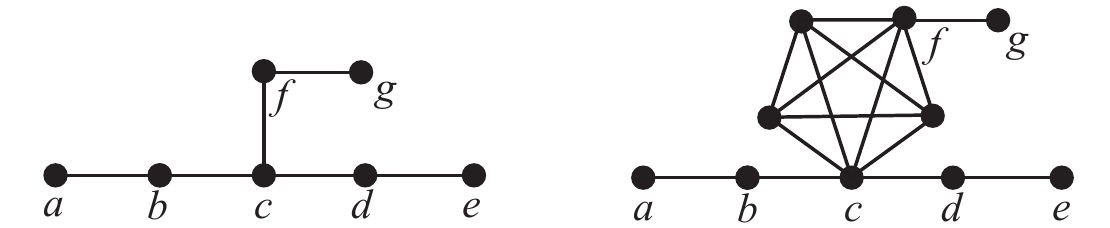}
\end{center}
\caption{The Seagull graphs $S_2$ (left) and $S_5$ (right).}
\label{seagulls}
\end{figure}

\begin{theorem} \label{seagull-theorem}
For positive integers $k$ and $j$ with $j>k>1$, the graph $S_k$ is not a $j$-MSP graph.
\end{theorem}

\begin{proof}
Label $S_k$ as in Figure~\ref{seagulls}.  Suppose by way of contradiction that $L$ is a $j$-lizard with MSP-graph $S_k$.  Label the scales $A$, $B$, $C$, $D$, $E$, $F$, and $G$ in $L$ corresponding to the vertices $a$, $b$, $c$, $d$, $e$, $f$, and $g$ in $S_k$.

By Lemma~\ref{parallel sides corollary}, the region $C-(B\cup D)$ has parallel sides.  If $B \cap C$ is not an end region of $C$, then there is another scale intersecting both $B$ and $C$.  Since this would contradict the construction of $L$, $B\cap C$ is a 2-end region of $C$, and similarly $D \cap C$ is a 2-end region of $C$.  Since $F \cap C$ doesn't intersect the interiors of $B$ or $D$, and by Lemma~\ref{end regions lemma} $F \cap C$ is not a 2-end region, $C-F$ is disconnected.  By Lemma~\ref{parallel sides lemma}, the sides $s_3$ and $s_4$ of $F$ incident with $C$ are parallel, and they must intersect one of the parallel sides of $C$, $s_1$.   Call these two vertices of $C \cup F$ $p_1$ and $p_2$.  Since $s_3$ and $s_4$ are parallel, the sum of the internal angles of $C \cup F$ is $3\pi$, as shown in Figure~\ref{seagull-contradiction}.  In what follows, assume $s_1$ is horizontal.

The angles at $p_1$ and $p_2$ are both reflex.  Suppose the angle at $p_1$ is $\theta_{j+i}$, for some integer $i$ with $1 \leq i \leq j-1$.  Then the angle at $p_2$ is $\theta_{2j-i}$.  If $p_1$ and $p_2$ are vertices of $L$, by Lemma~\ref{two-proto-scales-lemma} there is a point $x$ near $p_1$ in the intersection of at least $i(j-i)$ scales.

Suppose at least one of $p_1$ and $p_2$ is internal to $L$.  We may still form proto-scales at these points in the polygon $C \cup F$, even though they're not proto-scales of $L$.  For each pair of a proto-scale $q_1$ at $p_1$ and $q_2$ at $p_2$, extend these line segments down until they end at $\boundary(L)$.  If they intersect $B$ or $D$, then there is a scale intersecting both $B$ and $F$, or both $D$ and $F$, contradicting the construction of $L$.  So they must end at the other parallel side $s_2$ of $C$.  Now we extend $q_1$ and $q_2$ up until they intersect each other, forming a triangle $T$ with a base along $s_2$, as shown in Figure~\ref{seagull-contradiction}.  Say the left and right sides of $T$ are $s_5$ and $s_6$, respectively.

\begin{figure}[H]
\begin{center}
\includegraphics[width=.45\textwidth]{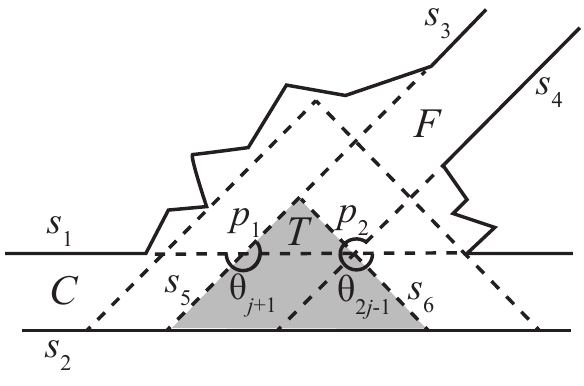}
\end{center}
\caption{The construction of the triangle $T$ in the proof of Theorem~\ref{seagull-theorem}.}
\label{seagull-contradiction}
\end{figure}

Now we enlarge $T$ by moving $s_5$ and $s_6$ left and right until they touch $\boundary(L)$, as shown in Figure~\ref{seagull-contradiction}.  This new triangle is a scale of $L$.  By the reasoning of the proof of Lemma~\ref{two-proto-scales-lemma}, there are $i(j-i)$ of these scales.  Since they are all triangles whose sides have different allowed directions, they are distinct.

Since $1 \leq i \leq j-1$, the expression $i(j-i)$ is minimized when $i=1$ or $i=j-1$, when $i(j-i)=j-1$.  So there are at least $j-1$ of these scales.  One of them may be $F$, but at least $j-2$ of these scales are distinct from $F$ and $C$, and intersect both $F$ and $C$.  Therefore there are at least $j-2>k-2$ vertices in $S_k$ adjacent to both $c$ and $f$.  Since $S_k$ has only $k-2$ such scales, this is a contradiction.

Thus $S_{k}$ is not a $j$-MSP graph.
\end{proof}

\begin{theorem} \label{seagull theorem 2}
$S_k$ is a $j$-MSP graph for $2\leq j \leq k$.
\end{theorem}

\begin{proof}
Again label $S_k$ as in Figure~\ref{seagulls}.  We construct a $j$-lizard $L$ with $j$-MSP graph $S_k$.  We start with four trapezoids and three parallelograms, arranged as in Figure~\ref{genseagullconstruction} on the left, which will be the scales $A$, $B$, $C$, $D$, $E$, $F$, and $G$, corresponding to vertices $a$, $b$, $c$, $d$, $e$, $f$, and $g$ in $S_k$.  The internal angles where these scales meet are all either $\theta_{j+1}$ or $\theta_{2j-1}$.  This yields a lizard with $j+5$ scales with corresponding $j$-MSP graph $S_j$, including the original seven scales $A$, $B$, $C$, $D$, $E$, $F$, and $G$, and an additional $j-2$ scales formed at the $\theta_{2j-1}$ angle between $F$ and $C$.

We then add $k-j$ additional scales as ridges on $F$, as shown on the right in Figure~\ref{genseagullconstruction}, for $j+5$ scales total.  These new scales correspond to vertices adjacent to every vertex in the clique, and the new $j$-lizard has $S_k$ as its $j$-MSP graph.
\end{proof}

\begin{figure}[H]
\begin{center}
\includegraphics[width=1\textwidth]{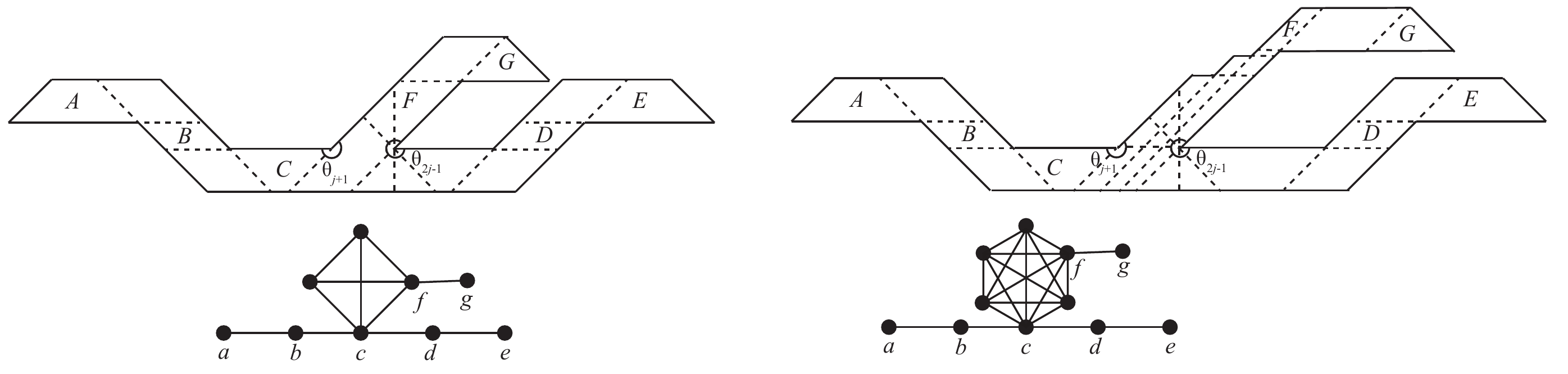}
\end{center}
\caption{The construction of $S_4$ and $S_6$ as 4-MSP graphs in the proof of Theorem~\ref{seagull theorem 2}.}
\label{genseagullconstruction}
\end{figure}

We define the family of \textbf{\textit{turtle graphs}}, $T_{n,k}$, to be the complete graph $K_n$ with $k$ additional pendant vertices, each adjacent to a distinct vertex of $K_n$.  The graph $T_{9,4}$ is shown on the left in Figure~\ref{turtle figure}.

\begin{theorem}
$T_{n,k}$ is a $k$-MSP graph for all $k>1$ and $n = \frac{1}{6}k^3 - \frac{7}{6}k + 3$.
\end{theorem}

\begin{figure}
\begin{center}
\includegraphics[height=3cm]{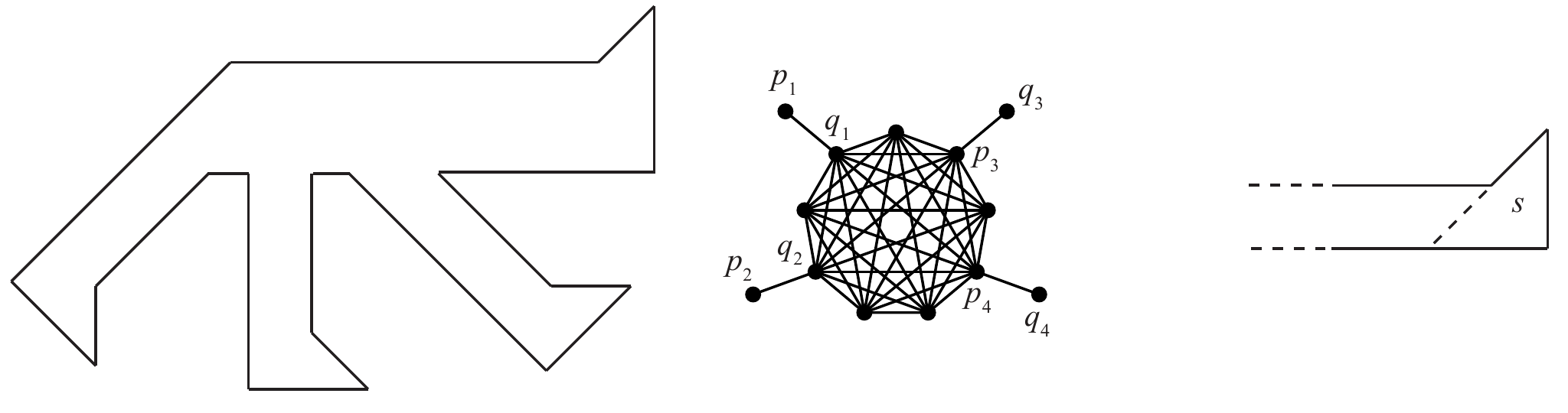}
\end{center}
\caption{On the left, a $4$-lizard with $4$-MSP graph $T_{9,4}$.  On the right, one leg of the turtle.}
\label{turtle figure}
\end{figure}

\begin{proof}
We construct a $k$-lizard $L$ with $k$-MSP graph $T_{n,k}$. We first construct a section of the $k$-lizard, called a leg, shown on the right in Figure~\ref{turtle figure}. We make each leg long enough so the scale $s$ at the end intersects only one other scale.  We use one wide leg as the body of the turtle, and attach $k-1$ additional legs to it so that the internal angle between each pair of adjacent legs is $\theta_{2k-1}$, and these angles fall on the same line, as shown in Figure~\ref{turtle figure}.  We also make the body leg wide enough so that the scales formed by the other legs all intersect.  Note that the internal angle between the first leg and the body is $\theta_{k+1}$. Then the first and last leg each create one scale in the clique. For an arbitrary leg $l_i$ between the first and last legs, any scale which contains part of the interior of $l_i$ disjoint from the body of the turtle must include the reflex vertices $v_{i-1}$ and $v_{i}$.  By Lemma~\ref{two-proto-scales-lemma}, there are $i(k-i)$ scales of this form. Hence, in total, there are
$$ 2+\sum_{i=1}^{k-2} i(k-i) = \frac{1}{6}k^3-\frac{7}{6}k+3=n$$ of these scales.  Since they all intersect, they form a clique of that size in the graph. The $k$ feet correspond to $k$ pendant vertices in the graph, so the lizard has $k$-MSP graph $T_{n,k}$ as claimed.
\end{proof}

\begin{theorem}
$T_{n,k+1}$ is not a $k$-MSP graph for any positive integers $n$ and $k$. \label{turtle-theorem}
\end{theorem}
\begin{proof}
Assume by way of contradiction that $L$ is a $k$-lizard whose $k$-MSP graph is $T_{n,k+1}$.  Label the pendant vertices of $T_{n,k+1}$ by $p_1,...,p_{k+1}$, and their respective adjacent vertices $q_1,...,q_{k+1}$.  Label the corresponding scales in $L$ by $P_1,...,P_{k+1}$, and $Q_1,...,Q_{k+1}$.  For all $1 \leq i \leq k+1$, let $R_i$ be the region of $Q_i$ disjoint from any other scale in $L$.  Since $p_1,...,p_{k+1}$ are pendant vertices, we may assume $R_1,...,R_{k+1}$ have non-empty interiors.

For all $1\leq i \leq k+1$ and some $j \neq i$, the vertices $p_i$, $q_i$, and $q_j$ satisfy the hypotheses of Lemma~\ref{parallel sides lemma}.  So by that lemma, the sides of $Q_i$ incident to $P_i$ are parallel.  By construction, these are sides of $R_i$.  Since $R_i$ is the region of $Q_i$ in no other scale of $L$, these parallel sides are also sides of $L$.

As this is the case for all $1 \leq i \leq k+1$, we have $k+1$ pairs of parallel sides of $L$. These pairs have directions in the set $\theta_1,...,\theta_k$, so by the pigeonhole principle, at least two pairs have the same direction. Let $Q_a$ and $Q_b$ be the scales containing these pairs, and let $s_1$ and $s_2$ be the parallel sides of scale $Q_a$ and $s_3$ and $s_4$ be the parallel sides of $Q_b$.  Assume without loss of generality that these sides are vertical, and consider their corresponding $x$-coordinates $x_1$, $x_2$, $x_3$, and $x_4$.  Again assume without loss of generality that $x_1<x_2$, $x_3<x_4$, and $x_1 \leq x_3$.

\noindent \textbf{Case 1. $x_2 \leq x_3$.}  In this case, the interiors of $Q_a$ and $Q_b$ don't intersect in $L$, which contradicts the construction of $T_{n,k+1}$.

\begin{figure}
\begin{center}
\includegraphics[height=4cm]{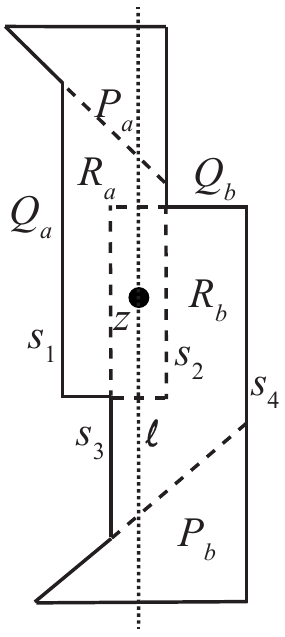}
\end{center}
\caption{The line $\ell$ in Case 2 of the proof of Theorem~\ref{turtle-theorem}.}
\label{turtle case}
\end{figure}

\noindent \textbf{Case 2. $x_3<x_2$.}  In this case, consider a point $z \in \interior(Q_a \cap Q_b)$.  Since $Q_a$ and $Q_b$ are convex, $z$ has $x$-coordinate between $x_3$ and $x_2$.  Again since $Q_a$ and $Q_b$ are convex, by Lemma~\ref{Ray Lemma} the vertical line $\ell$ through $z$ intersects the top and bottom of $R_a$ and $R_b$, as shown in Figure~\ref{turtle case}.  By Lemma~\ref{parallel sides lemma}, one of the top or bottom of $R_a$ borders $P_a$, and one of the top or bottom of $R_b$ borders $P_b$.  Then there is a scale containing a segment of $\ell$ that intersects $P_a$ and $P_b$, contradicting the construction of $L$.
\end{proof}

We can now combine $S_k$ and $T_{k,n}$ to construct a $k$-MSP graph which isn't a $j$-MSP graph for any $j\not=k$. We define the seagull-turtle graph $ST_{n,k}$ be the graph with vertex set $V(S_k) \cup V(T_{n,k}) \cup \{x\}$ for an additional vertex $x$, and with edge set $E(S_k) \cup E(T_{n,k}) \cup \{ex,xp_1\}$, using the labeling of Figure~\ref{genseagullconstruction} and Figure~\ref{turtle figure}.  The graph $ST_4$ is shown in Figure~\ref{seagull turtle figure}.

\begin{figure}
\begin{center}
\includegraphics[height=6cm]{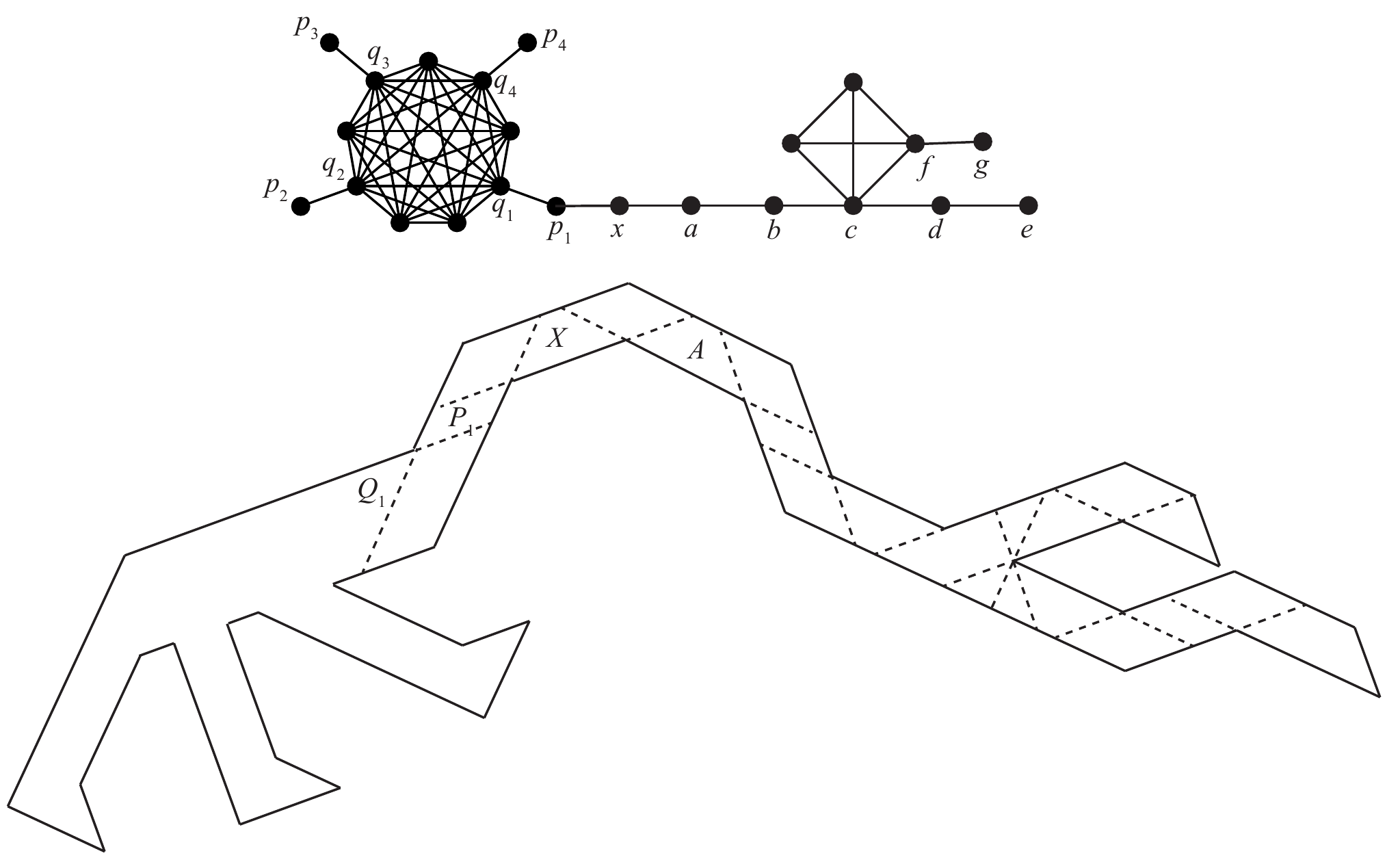}
\end{center}
\caption{The graph $ST_4$ and a $4$-lizard that has $ST_4$ as its $4$-MSP graph.}
\label{seagull turtle figure}
\end{figure}

\begin{lemma} \label{combining lemma}
For two graphs $G$ and $H$ and two vertices $x \in V(G)$ and $y \in V(H)$, let $J$ be the graph with vertex set $V(G) \cup V(H) \cup \{z\}$, and edge set $E(G) \cup E(H) \cup \{xz, zy\}$.  If $J$ is a $k$-MSP graph then $G$ and $H$ are $k$-MSP graphs.
\end{lemma}

\begin{proof}
Assume $J$ is a $k$-MSP graph, and let $L$ be a $k$-lizard with $k$-MSP graph $J$.  Let $X$ and $Z$ be the scales of $L$ corresponding to the vertices $x$ and $z$ in $J$.  By Lemma~\ref{parallel sides lemma}, the sides of $X$ incident to $Z$ are parallel, and the sides of $Z$ incident to $X$ are parallel.  Since there are no other scales of $L$ intersecting both $X$ and $Z$, by Lemma~\ref{two-proto-scales-lemma} $X \cup Z$ has at most one reflex angle.  Therefore $Z \cap X$ must be an end 2-region of $X$.  We may assume $X$ has a region in its interior disjoint from all other scales of $L$.

Consider the polygon $M$ formed by deleting the scales corresponding to vertices of $H \cup \{z\}$ from $L$, and let $X'=X \cap M$.  Since $Z$ is an end 2-region of $X$, and no other deleted scale intersects $M$,
$M$ is a $k$-lizard and $X'$ is a scale of $M$.  Furthermore, all other scales corresponding to $G$ in $L$ remain the same in $M$, so the $k$-MSP graph of $M$ is $G$.  A similar argument shows that $H$ is also a $k$-MSP graph.
\end{proof}

\begin{theorem}
Let $n=\frac{1}{6}k^3 - \frac{7}{6}k + 3$.  The seagull-turtle graph $ST_{n,k}$ is a $k$-MSP graph but not a $j$-MSP graph for $j\not= k$.
\end{theorem}

\begin{proof}
By Lemma~\ref{combining lemma}, since $S_k$ is not a $j$-MSP graph for $j<k$, $ST_{k,n}$ is not a $j$-MSP graph for $j<k$.  Again by Lemma~\ref{combining lemma}, since $T_{n,k}$ is not a $j$-MSP graph for $j>k$, $ST_{n,k}$ is not a $j$-MSP graph for $j>k$.  Finally, we construct a $k$-lizard with $k$-MSP graph $ST_{n,k}$ by combining the $k$-MSP representations of $S_k$ and $T_{n,k}$, as shown in Figure~\ref{seagull turtle figure}.
\end{proof}

\begin{corollary}
For integers $k>1$, $j>1$, and $j \neq k$, there is a $j$-MSP graph that is not a $k$-MSP graph.
\end{corollary}

The Venn diagram shown in Figure~\ref{venn results} summarizes the results of this section.  The one empty region, a graph that is $(k-1)$-MSP and $(k+1)$-MSP but not $k$-MSP, is still an open question.

\begin{figure}
\begin{center}
\includegraphics[height=5cm]{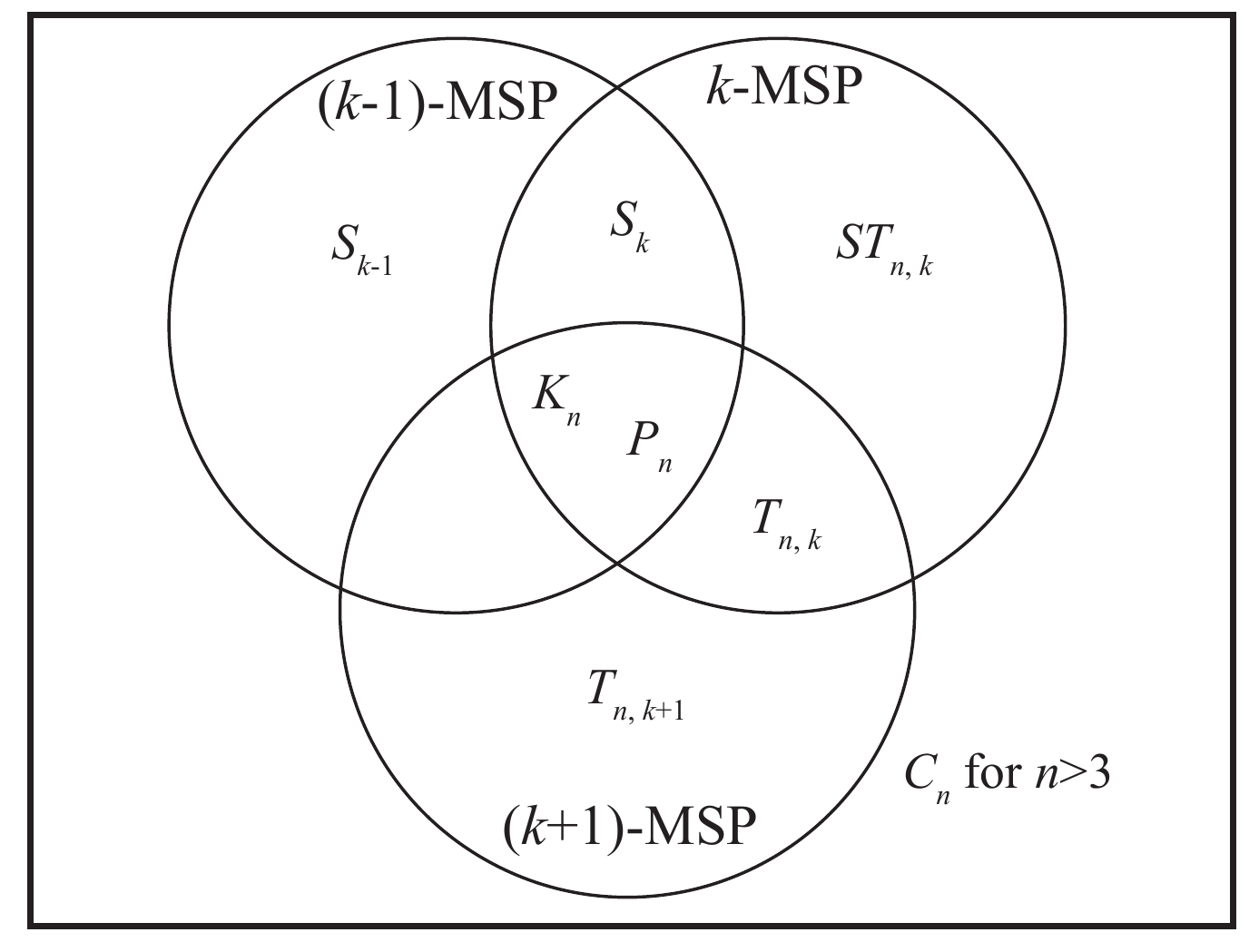}
\end{center}
\caption{A summary of the results of Section~\ref{separating examples}.}
\label{venn results}
\end{figure}

\section{Open Questions} \label{open questions section}

We conclude the paper with a list of open questions.

\begin{enumerate}
\item Are there 3-MSP graphs with induced cycles larger than 12, or 4-MSP graphs with induced cycles larger than 16?
\item Are there graphs which are $i$-MSP and $k$-MSP graphs but not $j$-MSP graphs for $i<j<k$?
\item In Lemma~\ref{combining lemma} we prove that in some instances induced subgraphs of $k$-MSP graphs are also $k$-MSP graphs, but this is not always the case, since for example large cycles are not $5$-MSP graphs but are induced subgraphs of $5$-MSP graphs by the construction in Figure~\ref{large induced cycles figure}.  Are there other conditions which determine when induced subgraphs of $k$-MSP graphs are also $k$-MSP graphs?
\item Our motivation for considering $k$-lizards was the $k$-snakes first defined in \cite{church2008snakes}, which have integer side lengths.  How would the results in this paper change with $k$-snake representations instead of $k$-lizard representations?  We might consider $k$-snakes with $k$-snake scales or $k$-snakes with $k$-lizard scales.  In particular, is there a $k$-MSP graph which is not representable with a $k$-snake? For a given $k$-lizard, it seems challenging to replace it with a $k$-snake representing the same graph.  Can this be done algorithmically?
\item If we allow $k$-lizards to have holes, most cycles are representable as $k$-MSP graphs.  What other graphs are representable in this way?
\item What if we allow intersections along boundaries of scales to count as edges in the graph?
\item What if we require additional restrictions on $k$-lizards?  For example, what graphs are representable with $k$-lizards with $j$ reflex angles for a given positive integer $j$?
\end{enumerate}

\section{Acknowledgements}

We thank Erin McNicholas and Richard Moy for their generous contributions of time and energy.

This project was funded by NSF grant DMS 1460982 as part of the Willamette Mathematics Consortium Research Experiences for Undergraduates program.

\bibliographystyle{plain}
\bibliography{1references}
\end{document}